\documentclass[11pt]{amsart}
\usepackage{amssymb}
\DeclareMathAlphabet{\mathantt}{OT1}{antt}{li}{it}
\DeclareMathAlphabet{\mathpzc}{OT1}{pzc}{m}{it}

\usepackage{tikz}
\usepackage{tikz-cd}
\usepackage[colorlinks=true,linkcolor=blue,urlcolor=blue,citecolor=blue, pagebackref]{hyperref}
\usepackage[alphabetic]{amsrefs}
\usepackage{soul}
\usepackage[margin=1.0in]{geometry}
\usepackage{xypic, verbatim,amscd,color}

\pagecolor{white}

\newcommand{\C}{\Bbb{C}}

\numberwithin{equation}{section}

\newcommand\op{\operatorname}

\newcommand{\stab}{\op{Stab}}

\newtheorem{theorem}{Theorem}[section]

\newtheorem{corollary}[theorem]{Corollary}
\newtheorem{question}[theorem]{Question}

\newtheorem{definition/lemma}[theorem]{Definition/Lemma}

\newcommand{\ag}{G(\mathcal{K})/G(\mathcal{O})}

\begin{document}
\title[PRV conjecture via cyclic convolution]{A proof of the refined PRV conjecture via the cyclic convolution variety}
\author{Joshua Kiers}
\begin{abstract}
In this brief note we illustrate the utility of the geometric Satake correspondence by employing the cyclic convolution variety to give a simple proof of the Parthasarathy-Ranga Rao-Varadarajan conjecture, along with Kumar's refinement. The proof involves recognizing certain MV-cycles as orbit closures of a group action, which we make explicit by unique characterization. In an appendix, joint with P. Belkale, we discuss how this work fits in a more general framework.
\end{abstract}
\maketitle
\section{Introduction}

We give a short proof of the Parthasarathy-Ranga Rao-Varadarajan conjecture first proven independently by Kumar in \cite{K} and Mathieu in \cite{M}. Our method extends to give a proof of Verma's refined conjecture which was first proven by Kumar in \cite{Kr}.

Let $\check G$ be a complex reductive group, whose representation theory we are interested in. (We reserve the symbol $G$ for the complex reductive Langlands dual group of $\check G$ because $G$ will be used more prominently in the proof, which goes through the geometric Satake correspondence.)
Fix a maximal torus $\check T$ and Borel subgroup $\check B$ of $\check G$.
Let $W$ be the Weyl group of $\check G$ (equivalently, of $G$). The statement of the original theorem is
\begin{theorem}[PRV conjecture]\label{playahata}
Let $\lambda, \mu$ be dominant weights for $\check G$ with respect to $\check B$, and let $w\in W$ be any Weyl group element. Find $v\in W$ so that $\nu:=v(-\lambda-w\mu)$ is dominant. Then
$$
\left(V(\lambda)\otimes V(\mu)\otimes V(\nu)\right)^{\check G}\ne (0).
$$
\end{theorem}

Kumar proved a refinement of this theorem in \cite{Kr} regarding the dimensions of the spaces of invariants. Let $W_\delta$ for any weight $\delta$ denote the stabilizer subgroup of $\delta$ in $W$. The stronger theorem is
\begin{theorem}[Refinement]
Let $\lambda, \mu, \nu, w$ be as above. Let $m_{\lambda, \mu, w}$ count the number of distinct cosets $\bar u\in W_\lambda \backslash W/ W_\mu$ such that $-\lambda - w\mu$ and $-\lambda - u\mu$ are $W$-conjugate (equivalently, $\nu$ can be written $q(-\lambda - u\mu)$ for some $q\in W$). Then
$$
\dim \left(V(\lambda)\otimes V(\mu)\otimes V(\nu)\right)^{\check G}\ge m_{\lambda, \mu, w}.
$$
\end{theorem}
In particular, since $m_{\lambda, \mu, w}\ge 1$ by definition, the second theorem implies the first.

We will use properties of a certain complex variety called the \emph{cyclic convolution variety}, whose definition we recall; see \cite{H}*{\S 2}, although our symmetric formulation is from \cite{Kam}*{\S1}. Let $G$ be the Langlands dual group to $\check G$ with dual torus $T$ and Borel subgroup $B$. Let $\lambda_i, i=1,\hdots,s$ be a collection of dominant weights for $\check G$ w.r.t. $\check B$; these induce dominant coweights of $G$ w.r.t. $B$.
Set $\mathcal{K} = \C((t))$, $\mathcal{O} = \C[[t]]$.
Each cocharacter $\lambda:\C^\times \to T$ induces an element $t^\lambda$ of $G(\mathcal{K})$; denote by $[\lambda]$ its image in $\ag$.
Recall that via the Chevalley decomposition any two points $L_1,L_2$ in $\ag$ give rise to a unique dominant coweight $\lambda$ of $T$ such that
$$
(L_1,L_2) = g([0],[\lambda])
$$
for some $g\in G(\mathcal{K})$; we write $\lambda = d(L_1,L_2)$ to convey this information concisely.

The cyclic convolution variety is
$$
\op{Gr}_{ G,c(\vec \lambda)} := \left\{(L_1,\hdots,L_s)\in \left(\ag\right)^s \mid  L_s = [0], d(L_{i-1},L_i) = [\lambda_i]~ \forall i   \right\},
$$
where we take $L_0$ to mean $L_s$. The maximum possible dimension of $\op{Gr}_{ G,c(\vec \lambda)}$ is $\langle \rho, \sum \lambda_i\rangle$, where $\rho$ is the usual half-sum of positive roots for $ G$, and via the geometric Satake correspondence (\cite{L,G,BD,MV}) the number of irreducible components of this dimension (if any) is equal to
$$
\dim \left(V(\lambda_1)\otimes \cdots \otimes V(\lambda_s)\right)^{\check G};
$$
see also \cite{H}*{Proposition 3.1}. 

Our task is therefore to produce irreducible components of the right dimension, which we find as $G(\mathcal{O})$-orbit closures of suitable points. These are in bijection with certain MV-cycles which we make explicit. As a corollary we obtain the following known result:

\begin{corollary}
Let $\lambda, \mu$ be dominant and $w\in W$ such that $\nu :=\lambda+w\mu$ is also dominant. Then the multiplicity of $V(\nu)$ inside $V(\lambda)\otimes V(\mu)$ is exactly $1$.
\end{corollary}

This is already known from a multiplicity theorem of Kostant; see \cite{Ko}*{Lemma 4.1} and \cite{Kumon}*{Corollary 3.8}. (It is also a consequence of Roth's theorem \cite{Roth} where $P_I=B$, $\bar G = \{1\}$, and the Schubert calculus equation is
$$
[\Omega_{w^{-1}}]\odot_0[\Omega_{e}]\odot_0[X_{w^{-1}}] = 1,
$$
using the notation found there.)

Our technique of producing components of the right dimension is not limited to the PRV setting; we illustrate this by an explicit example in Section \ref{con}.

Our proof of Theorem \ref{playahata} should be compared with the proof of \cite{Rz}*{Lemma 5.5}, where a geometric analogue of PRV is proved. There a one-sided dimension estimate on a fibre of the convolution morphism provides existence of components of the correct dimension, but the fibre component is not realized as the (closure of an) orbit under a group action; nor is the specific MV-cycle mentioned. The lower bound on number of components (yielding the refined version) is not made there. 

See also \cite{H2}*{Theorem 6.1}, where non-emptiness of the relevant variety (but not its dimension) is established, implying Theorem \ref{playahata} only in the case where $\lambda, \mu$ are sums of minuscule coweights. 

In an appendix, joint with P. Belkale, we describe the relationship of this work to a more
general question on the transfer of invariants between Langlands dual groups, with the PRV case corresponding to the inclusion of a maximal torus inside a reductive group.

\subsection{Acknowledgements} I thank Shrawan Kumar, Prakash Belkale, Joel Kamnitzer, and Marc Besson for helpful discussions and suggestions.

\section{Proof of the conjecture}

We will need some additional notation: let $\Phi$ denote the set of roots of $G$, and for $\alpha\in \Phi$ let $\alpha\succeq 0$ mean $\alpha$ is a positive root w.r.t. $B$ (likewise $\alpha\preceq 0$ means $-\alpha\succeq 0$).

\begin{proof}
{\bf Step 1} We claim that the cyclic convolution variety $\op{Gr}_{ G,c(\lambda, \mu, \nu)}$ is nonempty. Indeed, the point $x=([\lambda],[\lambda+w\mu],[0])$ satisfies
\begin{align*}
([0],[\lambda]) &=1 ([0],[\lambda])\\
([\lambda],[\lambda+w\mu]) &= t^\lambda w([0],[\mu])\\
([\lambda+w\mu],[0]) &= t^{\lambda+w\mu}v^{-1}([0],[\nu]).
\end{align*}

{\bf Step 2} Observe that any $\op{Gr}_{ G,c(\vec\lambda)}$ has a $ G(\mathcal{O})$-diagonal action on the left. We claim that the orbit $ G(\mathcal{O})x\subseteq \op{Gr}_{ G,c(\lambda, \mu, \nu)}$ is a finite-dimensional subvariety and has dimension $\langle \rho,\lambda+\mu+\nu\rangle$; this will conclude the proof, since the connectedness of $G(\mathcal{O})$ means $G(\mathcal{O})x$ is contained in an irreducible component of $\op{Gr}_{ G,c(\vec\lambda)}$ necessarily of dimension $\langle \rho, \lambda+\mu+\nu\rangle$.

For any integer $N>0$, let $K_N$ denote the kernel of the surjective group homomorphism
$$
G(\mathcal{O})\to G(\mathcal{O}/(t^N)).
$$
Observe that, for high enough $N\gg 0$, $K_N$ stabilizes the point $x$ (it suffices to embed $G$ into some $GL_m$ and examine matrix entries). Therefore $G(\mathcal{O})x$ has a transitive action by the finite-dimensional linear algebraic group $G(\mathcal{O}/(t^N))$.

The stabilizer $\stab_{G(\mathcal{O}/(t^N))}(x)$ is the image of
$$
\stab_{G(\mathcal{O})}(x)=G(\mathcal{O})\cap t^\lambda  G(\mathcal{O}) t^{-\lambda} \cap t^{\lambda+w\mu}  G(\mathcal{O}) t^{-\lambda-w\mu}\subseteq G(\mathcal{O})
$$
under the quotient; i.e., $\stab_{G(\mathcal{O}/(t^N))}(x) = \stab_{G(\mathcal{O})}(x)/K_N$.

By the orbit-stabilizer theorem,
$ G(\mathcal{O}/(t^N))x \simeq  G(\mathcal{O}/(t^N))/\stab_{G(\mathcal{O}/(t^N))}(x)$.
As \\
$G(\mathcal{O}/(t^N))/\stab_{G(\mathcal{O}/(t^N))}(x)$ is a smooth finite-dimensional variety,
we may calculate its dimension by the dimension of its tangent space
at the origin. For an arbitrary group scheme $H$ over $\C$, one takes $\op{Lie}(H)$ to mean the kernel of $H(\C[\epsilon]/(\epsilon^2))\xrightarrow{\epsilon\mapsto0} H(\C)$. Since Lie commutes with intersections (of subgroups of $G(\mathcal{K})$, see \cite{Milne}*{\S 10.c}), $\op{Lie}(\stab_{G(\mathcal{O})}(x))$ is
$$
\mathfrak{g}(\mathcal{O})\cap \op{Ad}_{t^\lambda}\mathfrak{g}(\mathcal{O})\cap \op{Ad}_{t^{\lambda+w\mu}}\mathfrak{g}(\mathcal{O})\simeq \mathfrak{h}(\mathcal{O})\oplus \bigoplus_{\alpha\in \Phi} t^{\max(0,\langle \alpha, \lambda\rangle,\langle \alpha, \lambda+w\mu\rangle)}\mathfrak{g}_\alpha(\mathcal{O})
$$
Thus in the quotient
$$
\op{Lie}(\stab_{G(\mathcal{O}/(t^N))}(x)) \simeq \mathfrak{h}(\mathcal{O}/(t^N))\oplus \bigoplus_{\alpha\in \Phi} t^{\max(0,\langle \alpha, \lambda\rangle,\langle \alpha, \lambda+w\mu\rangle)}\mathfrak{g}_\alpha(\mathcal{O}/(t^N));
$$
note that, for every $\alpha$, $\langle 0,\alpha\rangle\le N$ and $\langle \alpha, \lambda+w\mu\rangle\le N$ so that $K_N\subseteq \stab_{G(\mathcal{O})}(x)$. For the finite-dimensional affine group scheme $S:=\stab_{G(\mathcal{O}/(t^N))}(x)$, $\op{Lie}(S)$ is naturally identified with the tangent space of $S$ at the identity.  Therefore the $\C$-dimension of the tangent space $
\mathfrak{g}(\mathcal{O}/(t^N))/\op{Lie}(\stab_{G(\mathcal{O}/(t^N))}(x))
$
is
$$
\sum_{\alpha\in \Phi} \max(0,\langle \alpha, \lambda\rangle,\langle \alpha, \lambda+w\mu\rangle).
$$
The proof of the claim therefore reduces to the following calculation.

{\bf Step 3} We claim that $\langle \rho, \lambda+\mu+\nu\rangle = \sum_{\alpha\in \Phi} \max(0,\langle \alpha, \lambda\rangle,\langle \alpha, \lambda+w\mu\rangle).$ Let us examine the sum on the right in two parts, summing over $\alpha\preceq 0$ and $\alpha \succeq 0$ separately.

If $\alpha \succeq 0$, then $\max (0,\langle \alpha, \lambda\rangle,\langle \alpha, \lambda+w\mu\rangle) = \max (\langle\alpha, \lambda\rangle,\langle \alpha, \lambda+w\mu\rangle) $ due to the dominance of $\lambda$. Furthermore, $\langle \alpha, \lambda\rangle$ will be the bigger of the two unless $\langle \alpha, w\mu\rangle\ge 0$. Therefore
$$
\sum_{\alpha \succeq 0} \max (0,\langle \alpha, \lambda\rangle,\langle \alpha, \lambda+w\mu\rangle) = \sum_{\alpha \succeq 0} \langle \alpha, \lambda\rangle + \sum_{\begin{array}{c}\alpha \succeq 0\\ \langle\alpha, w\mu \rangle\ge 0 \end{array}} \langle \alpha, w\mu\rangle.
$$
The first sum on the RHS is clearly equal to $\langle 2\rho, \lambda\rangle$. As for the second sum, observe that $\langle \alpha, w\mu\rangle\ge 0 \iff \langle w^{-1}\alpha, \mu\rangle \ge 0$. As $\mu$ is dominant, this happens only when $w^{-1}\alpha \succeq 0$ or when $w^{-1}\alpha\preceq 0$ and $\langle w^{-1}\alpha, \mu\rangle =0$. The latter class of $\alpha$ doesn't contribute to the sum, so that second RHS term is equal to
$$
\sum_{\begin{array}{c}\alpha\succeq 0\\w^{-1}\alpha \succeq 0\end{array}} \langle \alpha, w\mu \rangle =
\sum_{\Phi^+\cap w\Phi^+} \langle \alpha, w\mu \rangle,
$$
where $\Phi^+$ denotes the set of positive roots.
As is well known (see for example \cite{Ku2}*{1.3.22.3}),
$
\sum_{\Phi^+\cap w\Phi^+} \alpha = \rho+w\rho.
$
Putting everything together so far, the original sum over $\alpha\succeq0$ yields $\langle 2\rho, \lambda\rangle+\langle \rho+w\rho, w\mu\rangle$.

If $\alpha \preceq 0$, then $\max (0,\langle \alpha, \lambda\rangle,\langle \alpha, \lambda+w\mu\rangle) = \max (0,\langle \alpha, \lambda+w\mu\rangle)$. Recall that $\lambda+w\mu = -v^{-1}\nu$; therefore the sum over $\alpha \preceq 0$ is
$$
\sum_{\begin{array}{c}\alpha \preceq 0\\ \langle \alpha, -v^{-1}\nu\rangle \ge 0\end{array}}  \langle \alpha, -v^{-1}\nu\rangle=
\sum_{\begin{array}{c}\alpha \succeq 0\\ \langle \alpha, v^{-1}\nu\rangle \ge 0\end{array}} \langle \alpha, v^{-1}\nu\rangle.
$$
As before, this equals $\langle \rho+v^{-1}\rho, v^{-1}\nu\rangle$. Finally, we conclude as desired that the dimension of the space in question is
\begin{align*}
\langle 2\rho, \lambda\rangle+\langle \rho+w\rho, &w\mu\rangle+ \langle \rho+v^{-1}\rho, v^{-1}\nu\rangle\\&=\langle \rho, \lambda + w\mu + v^{-1}\nu\rangle + \langle \rho, \lambda\rangle + \langle w\rho, w\mu\rangle + \langle v^{-1}\rho, v^{-1}\nu\rangle\\
&=0+\langle \rho, \lambda+\mu+\nu\rangle.
\end{align*}
\end{proof}

\section{Proof of the refinement}

\begin{proof}
Suppose $u\in W$ is such that $\nu = q(-\lambda-u\mu)$ for some $q\in W$. Then $x(u) := ([\lambda], [\lambda + u\mu],[0])$ satisfies
\begin{align*}
([0],[\lambda]) &=1 ([0],[\lambda])\\
([\lambda],[\lambda+u\mu]) &= t^\lambda u([0],[\mu])\\
([\lambda+u\mu],[0]) &= t^{\lambda+u\mu}q^{-1}([0],[\nu]);
\end{align*}
therefore $x(u) \in \op{Gr}_{ G,c(\lambda, \mu, \nu)}$ and $G(\mathcal{O})x(u)$ is a subvariety of $\op{Gr}_{G,c(\lambda, \mu, \nu)}$ of dimension $\langle \rho, \lambda+\mu+\nu\rangle$ for exactly the same reason as before.

{\bf Claim} If $x(u) = gx(u')$ for some $g\in G(\mathcal{O})$, then $\bar u = \bar u'\in W_\lambda\backslash W/W_\mu$.

{\bf Proof} Assume $x(u) = gx(u')$ for some $g\in  G(\mathcal{O})$. Fix $q,q'$ satisfying $\nu = q(-\lambda-u\mu) = q'(-\lambda-u'\mu)$. We are given that $g[\lambda] = [\lambda]$ and $g[\lambda+u'\mu] = [\lambda+u\mu]$. First we demonstrate that we can replace $g$ with an element of $G$. Recall from \cite{MV} that there is a map
$$
ev_0:Gr_\lambda \to G/P_\lambda,
$$
where $P_\lambda$ is the smallest parabolic containing $B^-$ and $L_\lambda$, where $B^-$ the Borel opposite to $B$ and $L_\lambda$ is the centralizer of $t^\lambda$ in $G$. The map is
given by $g(t)t^\lambda G(\mathcal{O})\mapsto g(0)P_\lambda$ and makes $Gr_\lambda$ an affine bundle over $G/P_\lambda$.

We find that $g(0)P_\lambda = P_\lambda$ by taking $ev_0$ of both sides of the equation $g[\lambda] = [\lambda]$; i.e., $g(0)\in P_\lambda$.
From $\nu = q(-\lambda - u\mu)$ we have $-w_0\nu = w_0q(\lambda+u\mu)$. The second equation can be formulated as
 $$
 g q'^{-1}w_0^{-1} [-w_0\nu] = q^{-1}w_0^{-1} [ -w_0\nu],
 $$
 which under $ev_0$ gives
 $
 g(0)q'^{-1}w_0^{-1} P_{-w_0\nu} = q^{-1}w_0^{-1} P_{-w_0\nu}.
 $

 We now attempt to replace $g(0)$ with a Weyl group element, as follows. Since $g(0)\in P_\lambda$,
 the double cosets
 $$
 P_\lambda q'^{-1}w_0^{-1}P_{-w_0\nu} = P_\lambda q^{-1}w_0^{-1}P_{-w_0\nu}
 $$
 agree, in which case
 $$
 W_\lambda q'^{-1}w_0^{-1} W_{-w_0\nu} = W_\lambda q^{-1}w_0^{-1}W_{-w_0\nu}
 $$
 by \cite{BT}*{Corollaire 5.20} (see also \cite{Kr}*{Lemma 2.2}).
 Writing $rq'^{-1}w_0^{-1} r' = q^{-1}w_0^{-1}$ for some $r\in W_\lambda, r'\in W_{-w_0\nu}$, observe that
\begin{align*}
\lambda+u\mu=q^{-1}w_0^{-1}(-w_0\nu)
= rq'^{-1}w_0^{-1}(-w_0\nu)
=r(\lambda+u'\mu)
=\lambda+ru'\mu;
\end{align*}
therefore $ru'\mu = u\mu$ and thus $ru'W_\mu = uW_\mu$. This gives $W_\lambda u'W_\mu = W_\lambda uW_\mu$ as desired.


So for any pair $\bar u, \bar u'$ distinct in $W_\lambda \backslash W/ W_\mu$ (such that $-\lambda-u\mu$ and $-\lambda - u'\mu$ are both conjugate to $\nu$), the orbits $G(\mathcal{O})x(u)$ and $G(\mathcal{O})x(u')$ must be disjoint. Each orbit $G(\mathcal{O})x(u)$ is irreducible, so the closure $\overline{G(\mathcal{O})x(u)}$ inside $\op{Gr}_{G,c(\lambda, \mu, \nu)}$ is an irreducible component of the same (top) dimension. Disjoint orbits necessarily give distinct (possibly not disjoint) irreducible components. Therefore the number of irreducible components of the top dimension of $\op{Gr}_{G,c(\lambda, \mu,\nu)}$ is at least $m_{\lambda, \mu, w}$, from which the theorem follows.
\end{proof}

\section{Relation to MV-cycles}
Here we recall the summary of the geometric Satake correspondence as presented in \cite{A}. Let $\mathcal{F}_\nu=\pi^{-1}([\nu])$ be the fibre of the natural projection map

\begin{center}
\begin{tikzcd}
\overline{\op{Gr}_{\lambda}}\tilde\times\overline{\op{Gr}_\mu}:= \{(aG(\mathcal{O}),bG(\mathcal{O}))\in \overline{\op{Gr}_\lambda}\times \overline{\op{Gr}_{\lambda+\mu}}\mid a^{-1}bG(\mathcal{O})\in \overline{\op{Gr}_\mu}\} \arrow[r,"\pi"] &
\overline{\op{Gr}_{\lambda+\mu}}
\end{tikzcd}
\end{center}
over $[\nu]$, where $\nu\preceq \lambda+\mu$.
Then the multiplicity of $V(\nu)$ inside $V(\lambda)\otimes V(\mu)$ is equal to the number of irreducible components of $\mathcal{F}_\nu$ of dimension $\langle \rho, \lambda+\mu-\nu\rangle$, the maximal possible dimension. (There is a $1-1$ correspondence between these irreducible components and those of top dimension in $\op{Gr}_{G,c(\lambda, \mu, -w_0\nu)}$.) According to \cite{A}*{Theorem 8}, the irreducible components of $\mathcal{F}_\nu$ of dimension $\langle \rho, \lambda+\mu-\nu\rangle$ are exactly the Mirkovi\'c-Vilonen cycles for $\overline{\op{Gr}_\lambda}$ at weight $\nu-\mu$ contained in $t^\nu\overline{\op{Gr}_{-\mu}}$.

As observed in \cite{A}, $\pi^{-1}([\nu]) = \overline{\op{Gr}_\lambda}\cap t^\nu \overline{\op{Gr}_{-\mu}}$. Therefore $\pi^{-1}([\nu])$ carries a natural $H:=G(\mathcal{O})\cap t^\nu G(\mathcal{O})t^{-\nu}$ action on the left. Note that $H$ is connected for the following reason: any $x(t)\in H$ has a path $x(st)$ connecting it to $x(0)$ as $s$ varies from $1$ to $0$. This gives a retraction of $H$ onto $P_\nu\subset H$, and $P_\nu$ is path-connected (as before, $P_\nu$ is the parabolic subgroup of $G$ containing $B^-$ and $L_\nu$). Therefore each irreducible component of $\pi^{-1}([\nu])$ is $H$-stable.

\begin{theorem}
Let $\lambda, \mu$ be dominant coweights. If $\nu = v(\lambda+w\mu)$ is dominant for some $v,w\in W$, then $V(\nu)$ appears in $V(\lambda)\otimes V(\mu)$ with multiplicity at least $1$. In fact, there is a unique MV-cycle $\overline{\op{Gr}_\lambda}$ at weight $\nu-\mu$ contained in $t^\nu\overline{\op{Gr}_{-\mu}}$ which contains $[v\lambda]$ (equivalently, contains $[qv\lambda]$ for all $q\in W_\nu$).
\end{theorem}

\begin{proof}
The point $[v\lambda]$ is clearly contained in $\pi^{-1}([\nu])$, since $[-v\lambda+\nu] = [vw\mu]\in \op{Gr}_{\mu}$.

{\bf Claim} $H.[v\lambda]$ has dimension $\langle \rho, \lambda+\mu-\nu\rangle$.

{\bf Proof} Exactly analogous to the previous dimension calculation.

Therefore the closure of $H.[v\lambda]$ gives an irreducible component of $\pi^{-1}([\nu])$ of the right dimension, so contributing to the multiplicity of $V(\nu)$ in $V(\lambda)\otimes V(\mu)$.

For uniqueness: if $A$ is any other irreducible component, $[v\lambda]\in A$ implies $H.[v\lambda]\subseteq A$, which forces $A = \overline{H.[v\lambda]}$.

Notably, any lift of any $q\in W_\nu$ to $G$ satisfies $q\in t^\nu G(\mathcal{O})t^{-\nu}$; therefore $[qv\lambda]\in H.[v\lambda]$.
\end{proof}

In similar style, the $m_{\lambda, \mu, w}$-many components produced as in the refinement are simply the $H$-orbits of the $[r\lambda]$s, where $\nu = r(\lambda+u\mu)$ as $u$ varies in $W_\lambda \backslash W/W_\mu$.

\begin{corollary}
If $\nu = \lambda+w\mu$ is dominant, the multiplicity of $V(\nu)$ in $V(\lambda)\otimes V(\mu)$ is exactly 1.
\end{corollary}

\begin{proof}
The cycle $A=\overline{H.[\lambda]}$ contributes $1$ to the multiplicity count. Since every MV-cycle of $\overline{\op{Gr}_\lambda}$ at weight $\nu-\mu$ contained in $t^\nu\overline{\op{Gr}_{-\mu}}$ must contain $[\lambda]$ and be $H$-stable, $A$ must be the only such cycle.
\end{proof}

\section{The converse fails}\label{con}

The entire basis of this work is a very strange phenomenon: for PRV triples $\lambda, \mu, \nu$, there exist irreducible components of $\op{Gr}_{G,c(\lambda, \mu, \nu)}$ \emph{containing a dense $G(\mathcal{O})$-orbit} (equivalently, there exist MV-cycles in $\mathcal{F}_\nu$ \emph{containing a dense $H$-orbit}). One could ask: given an irreducible top component of a cyclic convolution variety $\op{Gr}_{G,c(\lambda, \mu, \nu)}$ that contains a dense $G(\mathcal{O})$-orbit, is it true that $\lambda, \mu, \nu$ is a PRV triple? The answer turns out to be false:

\begin{theorem}
There exist $G,\lambda, \mu, \nu$ and an irreducible component $A\subset \op{Gr}_{G,c(\lambda, \mu, \nu)}$ of dimension $\langle \rho, \lambda+\mu+\nu\rangle$ such that
\begin{enumerate}
\item $A=\overline{G(\mathcal{O})x}$ for some $x$;
\item there are no elements $v,w\in W$ making $\nu = v(-\lambda-w\mu)$ true.
\end{enumerate}
\end{theorem}

\begin{proof}
Here is an example: take $G = SL_2$, $\lambda = \mu = \nu = \alpha^\vee$, the single positive coroot. Criterion (2) is easy to verify: $w\mu = \pm \alpha^\vee$ for any $w\in W$, and $v^{-1}\nu = \pm \alpha^\vee$ for any $v\in W$. But
$$
\pm \alpha^\vee = -\alpha^\vee\pm\alpha^\vee
$$
is not true for any choices of $+$,$-$.

As for (1): let $y = \left[\begin{array}{cc} 1 & t \\ 0 & 1\end{array}\right] t^{\alpha^\vee}$.

{\bf Claim} $x := ([\alpha^\vee],\bar y, [0])\in \op{Gr}_{G,c(\lambda,\mu,\nu)}$.

{\bf Proof} We have
\begin{align*}
([0],[\alpha^\vee])& = 1([0],[\alpha^\vee])\\
([\alpha^\vee],\bar y)&= t^{\alpha^\vee}\left[\begin{array}{cc} 0 & 1 \\ -1& t \end{array}\right]([0],[\alpha^\vee])\\
(\bar y, [0]) &= t^{\alpha^\vee} \left[\begin{array}{cc} 0 & 1 \\ -1& t \end{array}\right]  t^{\alpha^\vee} \left[\begin{array}{cc} 0 & 1 \\ -1& t \end{array}\right]([0],[\alpha^\vee]);
\end{align*}
the second line follows from
$$
y = \left[\begin{array}{cc} t & 1 \\ 0 & t^{-1} \end{array}\right] = \left[\begin{array}{cc} t &0 \\ 0&t^{-1} \end{array}\right] \left[\begin{array}{cc} 0&1 \\-1 &t \end{array}\right] \left[\begin{array}{cc}t &0 \\0 &t^{-1} \end{array}\right] \left[\begin{array}{cc} 1&0 \\t &1 \end{array}\right]
$$
and the third from
$$
\left[\begin{array}{cc} t &0 \\ 0&t^{-1} \end{array}\right] \left[\begin{array}{cc} 0&1 \\-1 &t \end{array}\right] \left[\begin{array}{cc}t &0 \\0 &t^{-1} \end{array}\right] \left[\begin{array}{cc} 0&1 \\-1 &t \end{array}\right] \left[\begin{array}{cc}t &0 \\0 &t^{-1} \end{array}\right] = \left[\begin{array}{cc} -t & 1 \\ -1 & 0 \end{array}\right].
$$

{\bf Claim} The dimension of $SL_2(\mathcal{O}).x$ is $\langle \rho, 3\alpha^\vee\rangle = \langle \alpha/2,3\alpha^\vee\rangle = 3$.

{\bf Proof} The stabilizer of $x$ has Lie algebra
$$
L = \mathfrak{sl}_2(\mathcal{O})\cap \op{Ad}_{t^{\alpha^\vee}} \mathfrak{sl}_2(\mathcal{O}) \cap \op{Ad}_{y} \mathfrak{sl}_2(\mathcal{O});
$$
we now try to express this vector space more explicitly.

Let $e,f,h$ be the standard basis of $\mathfrak{sl}_2(\C)$; then
$$
\mathfrak{sl}_2(\mathcal{O}) = e(\mathcal{O})\oplus h(\mathcal{O}) \oplus f(\mathcal{O})
$$
and
$$
\op{Ad}_{t^{\alpha^\vee}} \mathfrak{sl}_2(\mathcal{O}) = t^2e(\mathcal{O}) \oplus h(\mathcal{O})\oplus t^{-2}f(\mathcal{O}).
$$

Of course, $\op{Ad}_y\mathfrak{sl}_2 = \op{Ad}_{z}\op{Ad}_{t^{\alpha^\vee}} \mathfrak{sl}_2$,
 where $z = \left[\begin{array}{cc} 1 & t \\ 0 & 1\end{array}\right]$. One calculates
 \begin{align*}
 \op{Ad}_{z^{-1}}e = e; ~~~
 \op{Ad}_{z^{-1}}h = h+2te; ~~~
\op{Ad}_{z^{-1}}f = f-th-t^2e.
 \end{align*}

Let $X\in \mathfrak{sl}_2(\mathcal{O})\cap \op{Ad}_{t^{\alpha^\vee}}\mathfrak{sl}_2(\mathcal{O})$ be arbitrary: $X = p_ee+p_hh+p_ff$, where $\op{val}_t(p_e)\ge 2$, $\op{val}_t(p_h)\ge 0$, and $\op{val}_t(p_f)\ge 0$ (as usual, $\op{val}_t(0) = \infty$).

Now $X \in \op{Ad}_y \mathfrak{sl}_2(\mathcal{O})$ if and only if $\op{Ad}_{z^{-1}}X \in \op{Ad}_{t^{\alpha^\vee}} \mathfrak{sl}_2(\mathcal{O})$. As
$$
\op{Ad}_{z^{-1}}X = (p_e+2tp_h-t^{2}p_f)e+(p_h-tp_f)h+p_ff,
$$
this is if and only if $\op{val}_t(p_h)\ge 1$ (if $\op{val}_t(p_h)=0$, then the $e$-coefficient has $t$-valuation $1$ since $\op{val}_t(p_e-t^2p_f)\ge2$.)

Therefore $L=t^2e(\mathcal{O})\oplus th(\mathcal{O})\oplus f(\mathcal{O})$, in which case
$$
\mathfrak{sl}_2(\mathcal{O})/L \simeq \mathfrak{sl}_2(\mathcal{O})/t^2e(\mathcal{O})\oplus th(\mathcal{O})\oplus f(\mathcal{O}),
$$
and the latter has dimension $3$. So $\dim SL_2(\mathcal{O}).x = 3$.

The usual arguments then apply: $SL_2(\mathcal{O}).x$ is irreducible of maximal dimension; therefore its closure is an irreducible component.
\end{proof}

\section*{Appendix: A more general framework}
\begin{center}
by Prakash Belkale and Joshua Kiers\footnote{We thank N. Fakhruddin and S. Kumar for useful discussions.}
\end{center}

Let $H\to G$ be an embedding of complex reductive algebraic groups, and assume maximal tori and Borel subgroups are chosen such that $T_H\subseteq T_G$ and $B_H\subseteq B_G$. A priori, there is not a map $H^\vee \to G^\vee$ of Langlands dual groups; i.e., taking Langlands dual is not functorial. However, for any collection of coweights $\lambda_1,\hdots, \lambda_s$ for $T_H$ dominant w.r.t. $B_H$, there is a morphism of cyclic convolution varieties
$$
\Phi: \op{Gr}_{H,c(\vec\lambda)} \to \op{Gr}_{G,c(\vec\lambda')},
$$
where for each $i$, the ``transfer'' $\lambda_i':=w_i\lambda_i$ is the unique $G$-Weyl group translate of $\lambda_i$, viewed as a coweight of $T_G$, which is dominant w.r.t. $B_G$.
The morphism is just the embedding $H(\mathcal{K})/H(\mathcal{O})\to G(\mathcal{K})/G(\mathcal{O})$ in each factor; one easily verifies it is well-defined.

Therefore it is clear that $\op{Gr}_{H,c(\vec\lambda)}\ne \emptyset \implies \op{Gr}_{G,c(\vec\lambda')}\ne \emptyset$.

\begin{question}
Under what conditions on $H,G$ is true that
\begin{align}\label{imp}
(V(\lambda_1)\otimes \cdots\otimes V(\lambda_s))^{H^\vee}\ne (0) \implies
(V(\lambda_1')\otimes \cdots\otimes V(\lambda_s'))^{G^\vee}\ne (0)
\end{align}
for every tuple $(\lambda_1,\hdots, \lambda_s)$?

Equivalently, under what conditions on $H,G$ is it the case that if $\op{Gr}_{H,c(\vec\lambda)}$ has top-dimensional components then $\op{Gr}_{G,c(\vec\lambda')}$ does, too?
\end{question}

We note that consideration of mappings of ``dual groups" is an important theme in the Langlands program (cf. the functoriality conjecture \cite[Conjecture 3]{Gel}).

The weaker implication
\begin{align}\label{wi}
\exists N \text{ s.t. }(V(N\lambda_1)\otimes \cdots\otimes V(N\lambda_s))^{H^\vee}\ne (0)  \implies \exists N' \text{ s.t. }
(V(N'\lambda_1')\otimes \cdots\otimes V(N'\lambda_s'))^{G^\vee}\ne (0)
\end{align}
does hold; this is because the Hermitian eigenvalue cones for $H^\vee$ and $H$ are isomorphic, as are those for $G^\vee$ and $G$, see \cite{KLM}*{Theorem 1.8}, and there is a map between the Hermitian eigenvalue cones for $H$ and $G$ since there is a compatible mapping of maximal compact subgroups, see \cite{BK}. Therefore implication (\ref{imp}) always holds when $G$ is of type $A$ \cite{KT} or types $D_4,D_5,D_6$ \cites{KKM,Ki} by saturation. Here we note that $\op{Gr}_{G,c(\vec\lambda')}\ne \emptyset$  implies that $\sum \lambda_i'$ is in the coroot lattice for $G$ which equals the root lattice of ${G^\vee}$.

Setting $s=3$, the PRV theorem can be phrased as a partial answer to this question: if $H=T_G$ is a maximal torus of $G$, then (under no further conditions) implication (\ref{imp}) always holds. Indeed, $(V(\lambda_1)\otimes V(\lambda_2)\otimes V(\lambda_3))^{T^\vee}\ne (0)$ if and only if $\lambda_1+\lambda_2+\lambda_3 = 0$; therefore the $\lambda_i'$ satisfy $\lambda_1'+w\lambda_2'+v\lambda_3' = 0$ for suitable $w,v\in W$ and PRV says that $(V(\lambda_1')\otimes V(\lambda_2')\otimes V(\lambda_3'))^{G^\vee}\ne (0)$.

A series of instances where the implication (\ref{imp}) holds can be found in \cite[\S2]{HS}. In these examples $H$ is the subgroup of fixed points of a group $G$ under a diagram automorphism. Further, in each of these situations $H$ is of adjoint type.

When $H = PSL(2)$ and $G$ is arbitrary, implication (\ref{imp}) holds with no conditions. This follows from the linearity of the map $(\lambda_i)\mapsto (\lambda_i')$ when the $\lambda_i$ are each coweights of $SL(2)$ and from the special form of the Hilbert basis of the tensor cone for $SL(2)$: they are $(\omega,\omega,0)$ and permutations, so their transfers are $(\lambda',\lambda',0)$ for some $\lambda'$. Since $(N\lambda',N\lambda',0)$ have invariants for some $N$ by (\ref{wi}), $N\lambda'$ is self-dual; therefore $\lambda'$ is also.

When $H=PSp(4)$ (type $C_2$) and $G = PSp(4m)$, we have checked that the transfer property (\ref{imp}) holds. To do this, we establish that the transfer map on dominant weights is linear. Then we identify a finite generating set for the tensor semigroup for $PSp(4)$, using a result of Kapovich and Millson \cite{KM}. Finally we check the transfer property on this set. 

However, we can exhibit the failure of (\ref{imp}) when $H=SL(2)$ and $G = SO(5)$, the map being the standard $SL(2)$ embedding corresponding to the root $\alpha_1$. Therefore some conditions on $H,G$ must be necessary; perhaps is suffices to assume that $Z(H')$ maps into $Z(G)$ where $H'=[H,H]$ is the semisimple part of $H$, and  $Z(\cdot)$ denotes the center. This includes the PRV case (since $H'=1$), as well as any case where $H$ is of adjoint type; it furthermore excludes the counterexample with $SL(2)\subseteq SO(5)$.

\noindent
Department of Mathematics, University of North Carolina, Chapel Hill, NC 27599\\
{{email:   jokiers@live.unc.edu (JK)}}

\end{document}